\documentclass[10pt]{amsart}
\usepackage{listings,graphicx,amsmath,varioref,amscd,ulem, amssymb,color,bm,stmaryrd,amsthm,amsfonts,graphics,geometry,latexsym,pgf,pst-all} 

\theoremstyle{plain}
\usepackage{esint}
\usepackage{color}

\usepackage{amsthm}
\theoremstyle{plain}
\newtheorem{theorem}{Theorem}[section]
\newtheorem{proposition}[theorem]{Proposition}

\usepackage{geometry}
\usepackage[colorlinks=false]{hyperref}

\theoremstyle{definition}
\newtheorem{defin}[theorem]{Definition}

\newtheorem{remark}[theorem]{Remark}

\theoremstyle{remark}

\def\huzu{H^{1}_{0}(0,L)}
\def\w{W^{1,1}_0(0,L)}

\def\wm{W^{1,m}_0(0,L)}

\def\bk{\color{black}}

\def\re{\mathbb{R}}
\numberwithin{equation}{section}
 \makeatletter
\@namedef{subjclassname@2020}{%
  \textup{2020} Mathematics Subject Classification}
\makeatother

\def\dys{\displaystyle}

\newcommand{\car}[1]{\raise1pt\hbox{$\chi$}_{#1}}

\def\og{\leavevmode\raise.3ex\hbox{$\scriptscriptstyle\langle\!\langle$~}}
\def\fg{\leavevmode\raise.3ex\hbox{~$\!\scriptscriptstyle\,\rangle\!\rangle$}}

\author[D. Giachetti]{Daniela Giachetti}
\address[D. Giachetti]{Dipartimento di Scienze di Base e Applicate per l'Ingegneria, Sapienza Universit\`a di Roma, Via Scarpa 16, 00161 Roma, Italy 
	\\ {daniela.giachetti@gmail.com}}
	\author[P.J. Mart\'inez-Aparicio]{Pedro J. Mart\'inez-Aparicio}
\address[P.J. Mart\'inez-Aparicio]{Departamento de Matem\'aticas,
Universidad de Almer\'ia, Ctra. Sacramento s/n, La Ca\~{n}ada de San Urbano, 04120 Almer\'ia, Spain, pedroj.ma@ual.es}
	\author[F. Murat]{Fran\c{c}ois Murat}
\address[F. Murat]{Laboratoire Jacques-Louis Lions, Bo\^ite courrier 187, Sorbonne Universit\'e, 4 place Jussieu, 75252 Paris cedex 05, France, francois.murat@sorbonne-universite.fr}
\author[F. Petitta]{Francesco Petitta} 
\address[F. Petitta]{Dipartimento di Scienze di Base e Applicate per l'Ingegneria, Sapienza Universit\`a di Roma, Via Scarpa 16, 00161 Roma, Italy 
	\\ francesco.petitta@uniroma1.it}

\keywords{Elliptic problems, Singular terms, Divergence terms, Definition of a solution, Existence, Non-existence} \subjclass{35J75, 35B30, 35B35, 34A06, 34A12}

\begin{document}

\title[]{New definitions of a solution\break of a  one-dimensional elliptic equation\break with a singular first order divergence term}
\maketitle

\vspace{-1em}
\begin{center}
{\it Dedicato a Gioconda Moscariello per il suo settantesimo compleanno}
\end{center}

\begin{abstract}
In this paper, which is a continuation of our recent paper \cite{GMMP1}, we give two new definitions of a weak solution of the one-dimensional, elliptic, nonlinear singular problem which is formally written as
\begin{equation*}
\begin{cases}
\dys -\frac{d}{dx}\left(a(x) \frac{d u}{dx}\right)  = - \frac{d \phi (u) }{dx}- \frac{d g(x)  }{dx}& \text{in}\;(0,L),\\
u(0)=u(L)=0\,, & 
\end{cases}
\end{equation*}
where $\phi:\re\mapsto\re\cup\{+\infty\}$ is a singular function whose model is given by $\dys\phi(s)=\phi_\gamma (s)=\frac{1}{|s|^{\gamma}}$ with $\gamma>0$. In these two  definitions the solutions  belong respectively  to the space    $u\in W_0^{1,1}(0,L)$ with $\phi(u)\in L^1(0,L)$, and  to the space      $u\in W_0^{1,m}(0,L)$ with $\phi(u)\in L^m(0,L)$ for $m>1$.  In these frameworks we state and prove results  of  non-existence, existence,  and non-isolation of the   solutions.

\end{abstract}

\tableofcontents

\section{Presentation of the problem and of three possible definitions of weak solutions}\label{sec1}

In this  paper, which is a continuation of our recent paper \cite{GMMP1}, we will study  the  one-dimensional ($N=1$), elliptic, nonlinear singular  problem which is formally written as 
	
\begin{equation}
\label{pb1}
\begin{cases}
\dys -\frac{d}{dx}\left(a(x) \frac{d u}{dx}\right)  = - \frac{d \phi (u) }{dx}- \frac{d g(x)  }{dx}& \text{in}\;(0,L),\\
u(0)=u(L)=0\,, & 
\end{cases}
\end{equation}
where $u$ is the unknown, where  
\begin{equation}
\label{cond1}
N=1,\,\,\,\, L>0,  \,\,\mbox{ and } \,\,m\geq 1,
\end{equation}
 and where  the data ($a$, $g$, $\phi$)  satisfy
\begin{equation}\label{a1}
a\in L^{\infty}(0,L),\,\, \exists \,\, \alpha,\beta,\,\, 0<\alpha<\beta, \ \alpha\leq a(x)\leq  \beta\, \,\, \text{ a.e.}\,\, x\in(0,L),
\end{equation} 
\begin{equation}\label{g1}
g\in L^{m}(0,L), 
\end{equation}
\begin{equation}
\label{condphi1}
\begin{cases}
\phi:\re\mapsto\re\cup\{+\infty\},
\\\phi\,\, \mbox{ is continuous with values in } \re\cup\{+\infty\},\\
 \phi(s)<+\infty, \ \ \forall s\in\re,\,\, \ s\neq 0\,.
\end{cases}
\end{equation}

We are mainly interested  in the case where $\phi$ is singular at $s=0$,   i.e.  in the case where 
\begin{equation} \label{p3}\phi(0)=+\infty\,. \ \ \ \end{equation}

\noindent We will however also consider functions $\phi$ which do not satisfy \eqref{p3}, for instance when approximating a singular function $\phi$ which satisfies \eqref{condphi1} and \eqref{p3} by a sequence of functions $\phi_n$ which belong to $C^0(\re)$, and therefore satisfy \eqref{condphi1} but not \eqref{p3}.
\medskip 
  
The {\bf model case} for the function $\phi$ is the function $\phi_\gamma$ given by
\begin{equation}
\label{pM}
%\begin{cases}
\dys\phi_\gamma (s)=\frac{1}{|s|^{\gamma}}\ \ \text{with}\ \gamma>0,   
%\end{cases}
\end{equation}
which satisfies \eqref{condphi1} and \eqref{p3}. 

\bigskip 

 When looking at problem \eqref{pb1} in order to give it a   correct mathematical sense, a ``natural" framework is to look for a solution $u\in \w$ with $\phi(u)\in L^1(0,L)$, which allows one to give a correct mathematical meaning to the second line of \eqref{pb1} as well as a correct mathematical meaning, in the sense of distributions in $(0,L)$,  to every term of the first line of \eqref{pb1}, and, accordingly, to take $m= 1$ in order to have the datum $g$ in $L^1(0,L)$.  We therefore introduce in the present paper the following notion of solution (Definition $1$); this definition is new.
\begin{defin}[Definition $1$]
\label{defin}
 Assume that \eqref{cond1} holds true with 
 \begin{equation}
 \label{2.7bis}
 m=1,
 \end{equation} 
 and that  the data $(a,g,\phi)$ satisfy hypotheses \eqref{a1}-\eqref{condphi1}. We will say that $u$ is a {\it weak solution of problem \eqref{pb1} in the sense of Definition} $1$ if $u$ satisfies 
\begin{equation}\label{formdef1}
\begin{cases}\vspace{0.1cm}
\dys u\in \w,\  \phi(u)\in L^1(0,L),  \\ 
\dys \int_0^L a(x) \frac{d u}{dx} \frac{d z}{dx} dx =\dys  \int_0^L   \phi (u) \frac{d z}{dx} dx+\int_0^L   g(x)  \frac{d z}{dx}dx,\,\,\,\,
\forall z\in {\rm Lip}(0,L)\,\, \mbox{ with } \,\, z(0)=z(L)=0.
\end{cases}
\end{equation} 
 \hfill\qed
\end{defin}

The name ``Definition $1$"  (Definition \ref{defin}   above) is intended to remind the reader that this definition of the solution is based on the space $L^1(0,L)$ and on the space $W^{1,1}_{0}(0,L)$ which is built from it.

\bigskip 

On the other hand, we actually already studied in \cite{GMMP1} problem \eqref{pb1} in a different framework, in which we searched  for a solution $u\in \huzu$  with $\phi(u)\in L^{2}(0,L)$, when the  datum $g$ belongs to  $L^{2}(0,L)$. In that paper we defined a weak solution of problem \eqref{pb1} as follows: 

\begin{defin}[Definition $2$]\label{defin2}
Assume that \eqref{cond1} holds true with 
 \begin{equation}
 \label{2.7bis2}
 m=2,
 \end{equation} 
and that  the data $(a,g,\phi)$ satisfy hypotheses \eqref{a1}-\eqref{condphi1}. We will say that $u$ is a {\it weak solution of problem \eqref{pb1} in the sense of Definition $2$} if $u$ satisfies 
\begin{equation}\label{form-defi2}
\begin{cases}
\displaystyle u\in \huzu\,, \ \phi(u)\in L^2(0,L),  \\ \\ 
\dys \int_0^L a(x) \frac{d u}{dx} \frac{d z}{dx} dx =\dys  \int_0^L   \phi (u) \frac{d z}{dx} dx+\int_0^L   g(x)  \frac{d z}{dx}dx, \,\,\,\,\forall z\in  H^{1}_0 (0,L).
\end{cases}
\end{equation}
 \qed
\end{defin}

The name ``Definition $2$" (Definition \eqref{defin2}   above) is intended to remind  the reader that this definition is based on the space $L^2(0,L)$ and on the space $H_0^1(0,L)$ which is built from it. This definition was called ``Definition 2.6" in \cite{GMMP1}. 

\bigskip
\indent Let us also introduce in the present paper the following natural variant of Definitions $1$ and $2$; this definition is also new. 
\begin{defin}[Definition $m$]
\label{definm}
 Assume that \eqref{cond1} holds true with
 \begin{equation}
 \label{2.11bis}
 m>1,
 \end{equation}
 and that  the data $(a,g,\phi)$ satisfy hypotheses \eqref{a1}-\eqref{condphi1}. We will say that $u$ is a {\it weak solution of problem \eqref{pb1} in the sense of Definition} $m$  if $u$ satisfies 
\begin{equation}\label{formdefm}
\begin{cases}\vspace{0.1cm}
\dys u\in W^{1,m}_0 (0,L),\  \phi(u)\in L^m(0,L),  \\ 
\dys \int_0^L a(x) \frac{d u}{dx} \frac{d z}{dx} dx =\dys  \int_0^L   \phi (u) \frac{d z}{dx} dx+\int_0^L   g(x)  \frac{d z}{dx}dx,\,\,\,\,\forall z\in  W^{1,m'}_0 (0,L),
\end{cases}
\end{equation} 
with $m'$ the H\"older conjugate of $m$, namely $\dys m'=\frac{m}{m-1}$.

  \hfill\qed
\end{defin}

\smallskip 

Please note that Definition $2$ (Definition \eqref{defin2}   above) is the particular case of Definition $m$ (Definition \ref{definm} above) where $m=2$.
\smallskip

\begin{remark}\label{r1.4}
\noindent Observe that, for $m=1$,  Definition $1$ (Definition \ref{defin} above) and, for $m>1$, Definition $m$ (Definition \ref{definm} above) are equivalent to searching for $u$ which satisfies 
\begin{equation}\label{defin1dis}
\begin{cases}
\displaystyle u\in \wm,\,\, \phi(u)\in L^m (0,L),  \\ 
\displaystyle-\frac{d}{dx}\left(a(x) \frac{d u}{dx}\right)  = \displaystyle- \frac{d \phi (u)}{dx} \displaystyle- \frac{d g(x) }{dx} \ \  \text{in} \ \ \mathcal{D}'(0,L).
\end{cases}
\end{equation}
Observe also that in view of the one-dimensional setting that  we are considering in the present paper (see \eqref{cond1}), formula \eqref{defin1dis} is equivalent to searching for $u$ and $c$ which satisfy 
\begin{equation}\label{definODE}
\begin{cases}
\displaystyle u\in \wm\,, \ \phi(u)\in L^m (0,L),   \ \, c\in\re,\\ 
\displaystyle a(x) \frac{d u}{dx}   = \displaystyle   \phi (u) \displaystyle+  g  +c\ \ \ \text{in $\mathcal{D}'(0,L)$}.   
\end{cases}
\end{equation}
\noindent Note also that dividing the last line of \eqref{definODE} by $a(x)$ (recall that $a(x)\geq \alpha>0$, see \eqref{a1}) and then integrating from $0$ to $L$ proves that  in \eqref{definODE} the constant  $c$ is not arbitrary, but is given in terms  of $u$, $\phi(u)$, and $g$ by the formula 
\begin{equation}\label{definC}
c=\dys-\frac{\dys\int_0^L \frac{\phi(u)}{a(x)}\,dx+\dys\int_0^L \frac{g(x)}{a(x)}\,dx}{\dys\int_0^L \frac{1}{a(x)}\,dx}.
\end{equation}

Observe finally that if $u$ is a weak solution of problem \eqref{pb1} for the data $(a,g,\phi)$ in the sense of  Definition $1$ (Definition \ref{defin} above) or in the sense of Definition $m$ (Definition \ref{definm} above), or equivalently if $u$ is a solution of \eqref{defin1dis}, then $u$ is also a solution of these problems for the data $(a,g+C,\phi)$ for every $C\in\mathbb{R}$. Note also that if in \eqref{definODE} one replaces the data $g$ by $\overline{g}=g+C$ for some $C\in\mathbb{R}$, while keeping $\overline{u}=u$, the constant $c$ in formula \eqref{definC} is replaced by $\overline{c}=c-C$: in other terms when the solution $u$ is fixed, the second line of formula \eqref{definODE} does not define the datum $g$ and the constant $c$, but only the sum $g+c$, and one has $\overline{g}+\overline{c}=g+C+c-C=g+c$.

\qed
\end{remark}

\begin{remark}
\label{1.4bis}

\indent Formula \eqref{definODE} justifies the choice of the space $L^{1}(0,L)$ for the datum $g$ in Definition $1$ (Definition \ref{defin} above) and the choice of the space $L^m (0,L)$ for the datum $g$ in Definition $m$ (Definition \ref{definm} above).

\qed
\end{remark}

\bigskip 
\begin{remark}\label{mnq}
Let us also observe that if $1\leq m\leq q$, one has $L^{q}(0,L)\subset L^{m}(0,L)$, and that therefore, for any $g\in L^{q}(0,L)$,   any weak solution of problem \eqref{pb1} in the sense of Definition $q$ (Definition \ref{definm} above) is also a weak solution of problem \eqref{pb1} in the sense of Definition $m$ (Definition \ref{definm} above). In particular, (in the case where $m=1$ and $q=2$) if $g \in L^{2}(0,L)$, any weak solution of problem \eqref{pb1} in the sense of Definition $2$ is also a weak solution  of problem \eqref{pb1} in the sense of Definition $1$. 

\qed
\end{remark}

\bigskip
\section{Summary of the results obtained in \cite{GMMP1} in the framework of Definition $2$}\label{sec2}

Our goal in the present paper is to study Definition $1$ and Definition  $m$ (Definition ~\ref{defin} and Definition \ref{definm} above),   but let us first present a brief  summary of the results obtained in \cite{GMMP1} using Definition $2$ (Definition~\ref{defin2} above). 

\smallskip
Definition $2$ was motivated by the following important fact: take  $z=u$ as test function in the second line of \eqref{form-defi2}, and observe that one can prove (see \cite[Lemma 2.12]{GMMP1}\footnote{Result \eqref{zh} is formally explained by the following (formal) computation: define  the function  $\psi$ by 
$$
\psi(s)=\int_0^s \phi(t)\, dt\ \forall s\in\re; 
$$
then formally one has 
\begin{eqnarray*}
\dys \int_0^L \phi(z)\frac{dz}{dx}\, dx= \int_0^L \psi'(z)\frac{dz}{dx}\, dx\ =\int_0^L  \frac{d\psi(z)}{dx}\, dx = \psi (z(L)) - \psi(z(0))=0-0=0. 
\end{eqnarray*}

The above formal computation becomes correct when $\phi$ is assumed to belong to  
$L^1_{\rm loc}(\re) $, namely when $\dys \int_{-\delta}^{\delta}|\phi(s)|ds<+\infty$ for $\delta>0$, but Lemma 2.12 of \cite{GMMP1} does not require this assumption; it nevertheless requires the existence of a function $z$ which satisfies $z\in \huzu  \mbox{ and  }\phi(z)\in L^2(0,L)$.}
) 
that 
\begin{equation}\label{zh} 
\dys \text{for any}\ z\in \huzu \,\,\mbox{ with } \,\, \phi(z)\in L^2(0,L) \,\,\, \text{one has }\ \int_0^L \phi(z)\frac{dz}{dx}\, dx=0\,.
\end{equation}

If $u$ is a weak solution  of problem \eqref{pb1} in the sense of Definition $2$ (Definition~\ref{defin2} above), namely a solution of \eqref{form-defi2}, then  \eqref{zh} implies that one has 
\begin{equation}
\label{vf2bis}
\dys\int_0^L a(x)\frac{du}{dx}\frac{du}{dx}\, dx= \int_0^L g(x)\frac{du}{dx}\,dx,
\end{equation}
 which in turn implies the a priori estimate 
\begin{equation}
\label{propdubis}
\left\|\frac{du}{dx}\right\|_{L^2(0,L)}\leq \frac{1}{\alpha}\|g\|_{L^2(0,L)},
\end{equation}
and also, using Morrey's embedding $\huzu\subset L^{\infty}(0,L)$,
\begin{equation}
\label{propduibis}
\left\|u\right\|_{L^{\infty}(0,L)}\leq \sqrt{L} \left\|\frac{du}{dx}\right\|_{L^2(0,L)} \leq \frac{\sqrt{L}}{\alpha} \|g\|_{L^2(0,L)}.
\end{equation}

Therefore,   every possible weak solution of problem \eqref{pb1} in the sense of Definition $2$ (Definition~\ref{defin2} above) satisfies the a priori estimates \eqref{propdubis} and \eqref{propduibis}, i.e. $\huzu$ and $L^\infty(0,L)$ bounds which depend only on $L$, on the coercivity constant $\alpha$ of $a$, and on $\|g\|_{L^2(0,L)}$, but which do not depend  on the function $\phi$, which is only assumed to satisfy \eqref{condphi1}.  These a priori estimates are important results which are specific to the case $m=2$.

\smallskip

Starting from approximations of the function $\phi$ by functions $\phi_n$ which belong to $C^0(\mathbb{R})$, estimate \eqref{propdubis} allows one to obtain solutions $u_n$ of approximating problems which are bounded in $\huzu$ and to try to pass to the limit in $n$, with the hope of  proving that the limit $u$  of a subsequence of these $u_n$'s is a solution of \eqref{form-defi2}, namely  a weak solution  of problem \eqref{pb1}  in the sense of Definition $2$   (Definition~\ref{defin2} above). 

Let us stress that this approximation procedure also works in dimension $N>1$ thanks to an a priori estimate similar to \eqref{propdubis}, and that it allows one to prove that $\phi_n(u_n(x))$  converges to $\phi(u(x))$ for almost every $x$; but it does not allow one to pass to the limit in $\phi_n(u_n)$ in the sense of distributions. 

\medskip 

In contrast, in the one-dimensional case,  we were able to obtain in \cite[Lemma 3.1]{GMMP1}, a new a priori estimate of  $\phi_n(u_n)$ in $L^2(0,L)$ when the limit $u$ of a subsequence of the $u_n$'s  satisfies $u\not\equiv0$; this original a priori estimate is specific to the case where $N=1$.

\smallskip 
 This latest result seems to be an existence result, but actually it  only proves the following alternative: 
either $u=0$, or $u$ is a weak  solution of problem \eqref{pb1}  in the sense of Definition $2$ (Definition \ref{defin2} above).

\smallskip 
One might  think that the case where $u=0$ is exceptional, but this is not the case, as proved by the  following two non-existence results:
in the case where $\phi(0) = +\infty$, 
it  does not exist any weak solution of problem \eqref{pb1} in the sense of Definition $2$ (Definition \ref{defin2} above) if either the datum $g \in L^2(0,L)$ satisfies $g(x) \ge -M$ a.e. $x \in (0,L)$ for a given $M$ (see  \cite[Theorem 4.1]{GMMP1}), and/or if the singular function $\phi$ satisfies (see \cite[Theorem 4.4]{GMMP1}) 
\[
\int_0^\delta |\phi(t)| \, dt = \int_{-\delta}^0 |\phi(t)| \, dt = +\infty\ \ \text{for $\delta > 0$,}
\]
 which corresponds to the case $\gamma \geq 1$ in the model case where $\dys \phi(s) = \phi_{\gamma}(s)= \frac{1}{|s|^\gamma}$  with $\gamma > 0$.

We also proved in \cite[Proposition 7.3]{GMMP1} that in the model case where $\dys\phi(s) = \phi_{\gamma}(s)= \frac{1}{|s|^\gamma}$ with $0 < \gamma < 1$,  there  exists a large set of data $g \in L^2(0,L)$ for which problem \eqref{pb1} has a weak solution   in the sense of Definition $2$ (Definition \ref{defin2} above). 
Indeed, for any given finite number $n$ of points $x_i$ with
\[
0 < x_1 < \,\cdots\, < x_i < x_{i+1} < \,\cdots\, < x_n < L,
\]
we constructed a large set of functions $w \in H^1_0(0,L)$ which satisfy  $w(x_i) = 0$ for every $i$ and  $\phi(w)\in L^2(0,L)$;  defining then $g \in L^2(0,L)$ by $\dys g = a(x) \frac{dw}{dx} - \phi(w)$, we have built a large set of data  $g\in L^2(0,L)$ for which problem \eqref{pb1} admits the weak solution $w$  in the sense of Definition $2$ (Definition \ref{defin2} above).

In a similar way,  in the model case $\dys \phi(s) = \phi_{\gamma}(s)= \frac{1}{|s|^\gamma}$ with $0<\gamma <1$, we constructed in \break\cite[Theorem 7.5]{GMMP1},   for any given datum $g \in L^2(0,L)$ and for any $\delta$ with $0 < \delta < L$,  a datum $\hat{g}_\delta \in L^2(0,L)$ such that $\hat{g}_\delta = g$ on $(0, L-\delta)$ for which problem \eqref{pb1} has a weak solution   in the sense of Definition $2$ (Definition \ref{defin2} above).
\medskip

We also proved in  \cite[Proposition 8.1]{GMMP1} that every weak solution $u$ of problem \eqref{pb1} in the sense of Definition  $2$ (Definition~\ref{defin2} above) is not isolated, or in other terms that  this solution $u$ is the weak limit in $\huzu$ of solutions $u_n$, with $u_n\not=u$,  of approximating problems with data $(a,g_n,\phi_n)$, for any a priori given sequence $\phi_n$ of ``reasonable approximations" (see Definition~\ref{3.7} below) of the singular function  $\phi$, provided that the data $g_n$ are conveniently chosen. On the other hand, we also proved in  \cite[Proposition 8.3]{GMMP1} that exactly in the same context, namely for the same a priori given sequence $\phi_n$ of ``reasonable approximations" of  $\phi$,   one can choose the data $g_n$ in such a way that the solutions $u_n$ of the approximating problems tend to zero weakly in $\huzu$.
\medskip

Let us finally mention that in  \cite[Section 5]{GMMP1} we studied the following singular ODE: for any given function $\phi$ which satisfies \eqref{condphi1} and possibly \eqref{p3}, and for any given datum  $h\in L^2(0,L)$, search for $v$ which satisfies 
\begin{equation}\label{ODE}
\begin{cases}
\dys v\in H^1(0,L),\,\, \phi(v)\in L^2(0,L),\\
\dys a (x) \frac{dv}{dx}   =  {\phi} (v) + h(x) \,\, \text{ in}\,\, \mathcal{D}'(0,L))\,,\\
\dys v(0)=0. & 
\end{cases}
\end{equation}

\noindent The ODE \eqref{ODE} is strongly related to problem \eqref{pb1}, since every solution $v$ of the ODE \eqref{ODE} is a weak solution of problem \eqref{pb1} with $g=h$ in the sense of Definition $2$ (Definition~\ref{defin2} above) if and only if $v(L)=0$.

We proved the existence of a solution of the ODE \eqref{ODE}, as well as an a priori estimate and a stability result. Under further hypotheses on $\phi$ and $h$, we also proved the positivity of any solution, and comparison and uniqueness results; see \cite[Subsection 5.3]{GMMP1}  for a synthesis of these results, which are new when the function $\phi$ is singular.
\medskip

We then used these results obtained on the ODE \eqref{ODE} to prove a multiplicity result (see \cite[Theorem 6.1]{GMMP1}) of the weak solutions of problem \eqref{pb1} in the sense of Definition $2$ (Definition~\ref{defin2} above).

The results obtained on the ODE \eqref{ODE}  also allowed us to generalize the results stated above in the model case $\phi(s)=\phi_\gamma (s)=\frac{1}{|s|^{\gamma}}$ with $0<\gamma<1$ to the general case of a function $\phi$ which satisfies \eqref{condphi1} and \eqref{p3} as well as 
\begin{equation}\label{inte}
\int_{-\delta}^{\delta}|\phi(t)|dt<+\infty\ \ \text{for $\delta>0$}. 
\end{equation}

\vskip1cm 
\section{New results  obtained in the frameworks of Definition $1$  and of Definition $m$}
\label{sec3}

This section is dedicated to state and prove results in the framework of Definition $1$ (Definition~\ref{defin} above) and of Definition $m$ (Definition~\ref{definm} above).

\smallskip
In Subsection \ref{sub3.1} we state and prove in Theorem \ref{t3.2} below that if one assumes that 
\begin{equation} \label{m-1} 
\dys m>1,\ \ \ \int_0^\delta|\phi (t)|^{m-1}\,dt= \dys\int_{-\delta}^0|\phi(t)|^{m-1}\,dt=+\infty \ \ \text{for   $\delta>0$},
 \end{equation}
then it does not exist any weak solution of problem \eqref{pb1} in the sense of Definition  $m$ (Definition \ref{definm} above). 

\smallskip

In contrast, in Subsection \ref{sub3.2}, we state and prove in Theorem \ref{t3.4} below that in the  model case \eqref{pM} where  \[
\phi(s) = \phi_\gamma(s) = \frac{1}{|s|^\gamma} \quad \text{with } \gamma > 0,
\]
when  $m$ and $\gamma$ satisfy either
\begin{equation} \label{3.1}
 m=1, \ \ \  \gamma>0,
\end{equation}
or
\begin{equation} \label{3.2}
m>1, \ \ \ 0 < \gamma < \frac{1}{m-1},
\end{equation}
there exists (a large set of) data $g\in L^{m}(0,L)$ for which problem \eqref{pb1} has at least  a weak solution in the sense of Definition $1$ (Definition \ref{defin} above) if $m=1$, or in the sense of Definition $m$ (Definition \ref{definm} above) if $m>1$.

Since in the model case \eqref{pM} condition \eqref{m-1} is equivalent to 
\begin{equation} \label{3.1bis}
m>1, \ \ \   \gamma \geq  \frac{1}{m-1},
\end{equation}
and since Theorem \ref{t3.4} below  holds true for every $\gamma>0$ when $m=1$, combining  the results of Theorem \ref{t3.2} and Theorem \ref{t3.4} proves that, in the model case \eqref{pM}, conditions    \eqref{3.2} and \eqref{3.1} are necessary and sufficient  in order to ensure the existence of (a large set of) data $g\in L^{m}(0,L)$ for which problem \eqref{pb1} has at least a weak solution in the sense of Definition $m$ (Definition \ref{definm} above) when $m>1$,  and in the sense of Definition $1$ (Definition \ref{defin} above) when  $m=1$. 

 In the case of a general singular function $\phi$ which only satisfies \eqref{condphi1} and \eqref{p3},  we are unfortunately unable to prove that similarly, condition  
\begin{equation}\label{nom-1}
\dys m>1\ \ \ \int_0^\delta|\phi (t)|^{m-1} \,dt<+\infty \ \ \text{and} \ \ \dys\int_{-\delta}^0|\phi(t)|^{m-1} \,dt<+\infty \ \ \text{for   $\delta>0$}, 
\end{equation}
(which is the counterpart of \eqref{m-1} and the analogue of \eqref{3.2} in the general case) is necessary and sufficient  in order to ensure the existence of (a large set of) data $g\in L^{m}(0,L)$ for which problem \eqref{pb1} has at least a weak solution in the sense of Definition $m$ (Definition \ref{definm} above). 

The results of Theorems~\ref{t3.2} and \ref{t3.4} below therefore show that the range of those $\gamma$'s   for which Definition $m$ (Definition \ref{definm} above) of a weak solution of problem \eqref{pb1} is meaningful varies with the value of $m$, which is surprising. Observe also that when $m=1$ there is no restriction on $\gamma$ in order for Definition $1$ (Definition \ref{defin} above) to be meaningful.

\smallskip 

Finally in Subsection \ref{sub3.3} we will state and prove (see Proposition \ref{stab} below) that every weak solution of problem \eqref{pb1} in the sense of Definition $1$ (Definition \ref{defin} above) or in the sense of Definition $m$ (Definition \ref{definm} above) is not isolated. 

\subsection{Two results of non-existence of weak solutions of problem \eqref{pb1} in the sense of Definition~$m$ or of Definition~$1$}\label{sub3.1}
\mbox{ }

In this subsection, we prove two results of non-existence which are of very different kinds. 
\smallskip

The first one (Theorem \ref{t3.2}  below) is concerned with the case where, for $m>1$, the function $\phi$ is too singular for   problem \eqref{pb1} to have a weak solution in the sense of Definition $m$ (Definition \ref{definm} above).

\begin{theorem}\label{t3.2}
Assume that \eqref{cond1} holds true, that $m>1$, and that the data $(a, g, \phi)$ satisfy \eqref{a1}, \eqref{g1}, \eqref{condphi1}, and \eqref{p3}. If the singular function $\phi$ satisfies  \eqref{m-1}, it does not exist any weak solution of problem \eqref{pb1} in the sense of Definition $m$ (Definition \ref{definm} above).
\end{theorem}

\medskip

The   result of Theorem \ref{t3.2} was proved in \cite[Theorem 4.4 and Remark 4.5]{GMMP1}, in the case where  $m=2$.

\begin{proof}[Proof of Theorem \ref{t3.2}] In order to prove Theorem \ref{t3.2} we will   prove that when \eqref{m-1}  
holds true, it does not exist any function $\hat{u}$ such that
\begin{equation}\label{3000}
\hat{u} \in W_0^{1,m}(0,L) \quad \text{with }\quad \phi (\hat{u})  \in L^m(0,L);\end{equation}
 this non-existence result immediately implies the non-existence result of Theorem \ref{t3.2}. 
 
 \smallskip
  
 \indent Let us begin with  the model case where  the singular function $\phi$ is given by $\dys\phi(s) = \phi_\gamma(s) = \frac{1}{|s|^\gamma}$. In this case the proof is particularly simple. 
 
 We have to prove that  it does not exist any function $\hat{u}$ such that
  \[
\hat{u} \in W_0^{1,m}(0,L) \quad \text{with }\quad \phi_\gamma(\hat{u})=\frac{1}{|\hat{u}|^{\gamma}} \in L^m(0,L). \] 

\noindent For that, recall that  Morrey's inequality (which is an immediate consequence of H\"older's inequality in dimension $N=1$) states that for $m > 1$ one has, with $\dys m'=\frac{m}{m-1}$,  
\[
\text{if } \hat{u} \in W^{1,m}(0,L), \text{ then } \hat{u} \in C^{0,\frac{1}{m'}}([0,L]),
\]
with 
\[
|\hat{u}(x)-\hat{u}(y)|\leq \left\|\frac{d \hat{u}}{dx}\right\|_{L^{m}(0,L)}|x-y|^{\frac{1}{m'}}\ \ \forall\ x\in[0,L]\ \forall\ y\in[0,L],
\]
\noindent which implies that when  $\hat{u} \in W_0^{1,m}(0,L)$, one has 
\[
|\hat{u}(x)| \le \left\|\frac{d \hat{u}}{dx}\right\|_{L^{m}(0,L)}x^{\frac{1}{m'}}\quad \forall x \in [0,L],
\]
so that, since 
$
\dys\gamma \ge \frac{1}{m-1},
$
one has
\[
\int_{0}^{L} \phi_\gamma(\hat{u})^m dx = \int_{0}^{L} \frac{1}{|\hat{u}|^{\gamma m}} dx \ge \dys \frac{1}{\dys \left\|\frac{d \hat{u}}{dx}\right\|_{L^{m}(0,L)}^{\gamma m}} \int_{0}^{L} \frac{1}{x^{ (m-1)\gamma}} dx = +\infty.
\]

\smallskip 
Therefore $\phi_\gamma(\hat{u})$ does not belong to $L^m(0,L)$ when $\hat{u}$ belongs to $W_0^{1,m}(0,L)$ and when
$
\dys \gamma \ge \frac{1}{m-1},
$
which proves the desired result (and Theorem \ref{t3.2}) in the model case where $\dys\phi(s) = \phi_\gamma(s) = \frac{1}{|s|^\gamma}$. 

\medskip

In the general case of a singular function  $\phi$ which satisfies  \eqref{condphi1}, \eqref{p3}, and \eqref{m-1}, the proof follows along  the lines of the proof of \cite[Proposition 4.8]{GMMP1}. 

\medskip 

We claim on the first hand  that if $m>1$ and if 
  the singular function  $\phi$   satisfies  \eqref{condphi1}, \eqref{p3}, and 
\begin{equation} \label{p+} 
\dys\int_0^\delta|\phi (t)|^{m-1}\,dt=+\infty  \ \ \text{for   $\delta>0$},
 \end{equation}
 and if  $\hat{u}$ satisfies \eqref{3000},   
  one has 
\begin{equation}
\label{412<0}
\hat{u}(x) \leq 0 \quad \forall x\in[0,L],
\end{equation}
and we claim  on the second hand that if  $m>1$ and if 
  the singular function  $\phi$   satisfies  \eqref{condphi1}, \eqref{p3}, and
\begin{equation} \label{p-} 
  \dys\int_{-\delta}^0|\phi(t)|^{m-1}\,dt=+\infty \ \ \text{for some $\delta>0$},
 \end{equation} and if  $\hat{u}$ satisfies \eqref{3000},   
  one has 
\begin{equation}
\label{412>0}
\hat{u}(x) \geq 0  \quad \forall x\in[0,L].
\end{equation}
 
\indent Therefore if the singular function $\phi$ satisfies \eqref{condphi1}, \eqref{p3}, and \eqref{m-1}, namely \eqref{condphi1}, \eqref{p3}, \eqref{p+}, and \eqref{p-}, and if ${u}$ satisfies \eqref{3000}, the claim implies that $u \equiv 0$, in contradiction with \eqref{3000}  since $\phi(0)\equiv +\infty$ does not belong to $L^{m}(0,L)$. This contradiction proves that it does not exist any $\hat{u}$ which satisfies \eqref{3000} when the singular function $\phi$ satisfies \eqref{condphi1}, \eqref{p3}, and \eqref{m-1}, and proves Theorem \ref{t3.2}. 
 
%We will prove that hypothesis \eqref{p+}  implies \eqref{412<0}. In  the case where \eqref{p-} is assumed instead of \eqref{p+}, the proof of \eqref{412>0} analogue. Also observe that, by \eqref{p3}, we may assume without loosing generality that $\phi\geq 0$ for $x\in [-\delta,\delta]$. 

\medskip 

Let us prove the first part of the claim,  namely the fact that  \eqref{condphi1}, \eqref{p3},  \eqref{p+}, and \eqref{3000} imply \eqref{412<0}. The proof of the second part of the claim is similar. 

Observe first that in view of \eqref{condphi1} and \eqref{p3}, one can  choose $\delta$ such that 
\begin{equation}
\label{3001}
\phi(s) >0  \quad \forall s\in[0,\delta].
\end{equation}

Assume  now  by contradiction with  \eqref{412<0} that    \eqref{condphi1}, \eqref{p3},  \eqref{p+}, and \eqref{3000} hold true but that there exists  some $x_0$ such that \begin{equation}\label{3002} 0<x_0<L \ \ \text{with} \ \ \hat{u}(x_0)>0. \end{equation}

Define $y_0$ by 
\begin{equation}
\label{411bis} 
y_0=\inf\{x\in[0,x_{0}] \ \text{such that}\  \hat{u}(z)>0 \ \text{for all}\  z\in(x,x_{0}]\}
\end{equation}
(one could also  consider $y_1$ defined by  $y_1=\sup\{x\in [x_0,L] \ \text{such that}\  \hat{u}(z)>0 \ \text{for all}\  z\in [x_{0},x)$). Observe that $y_0$ is well defined  since the set $\{x\in [0,x_0] \ \text{such that}\  \hat{u}(z)>0 \ \text{for all}\  z\in(x,x_{0}]\}$ is not empty: indeed $\hat{u}$ is continuous and in view of \eqref{3002} there exists some $\tau>0$ such that $\hat{u}(x)>0$ on $[x_0-\tau,x_0+\tau]$.

On the other hand,   since $\hat{u}(x)>0$ for every $x\in(y_0,x_0]$, one has  $\hat{u}(y_0)\geq 0$, but  actually one has 
\begin{equation}
\label{414bis}
\hat{u}(y_0)=0:
\end{equation}
indeed  if $\hat{u}(y_0)>0$, it would  exist some  $\tau>0$ such that $\hat{u}(x)>0$ on $[y_0-\tau,y_0+\tau]$, a contradiction with the definition \eqref{411bis} of $y_0$.

\medskip 

\indent Fix now some $\delta>0$ such that \eqref{3001} is satisfied,  and define the function  $\hat{\psi}_{\delta}: \ ]0,+\infty[\mapsto\re$ by
$$
\hat{\psi}_{\delta}(s)=\int_s^\delta \phi(t)^{m-1} dt\quad \forall s>0,
$$
or equivalently by
$$
\hat{\psi}_{\delta}(\delta)=0,\quad \hat{\psi}'_\delta(s)=-\phi(s)^{m-1}\quad \forall s>0. 
$$
 
 \noindent Then \eqref{p+} is equivalent to 
\begin{equation}
\label{451}
\hat{\psi}_{\delta}(s)\to+\infty \quad \mbox{ as } s\to 0,\,s>0.
\end{equation}

Recalling that $\hat{u}$ is continuous, we define    for $\eta$ such that  $0<\eta<x_0-y_0$ the two real numbers $\underline{u}_\eta$ and $\overline{u}$ by
$$
\underline{u}_\eta=\min_{x\in[y_0+\eta,x_0]}u(x),\quad \overline{u}=\max_{x\in[y_0,x_0]}u(x),
$$
and we observe that 
$$
\forall \eta\quad \text{such that}\quad 0<\eta<x_0-y_0, \mbox{ one has } \,\, 0<\underline{u}_\eta\leq \hat{u}(x)\leq \overline{u}<+\infty\quad \forall x\in [y_0+\eta,x_0].
$$
We also define   the two real numbers $\underline{\phi}$ and $\overline{\phi}_\eta$ by 
$$
\underline{\phi}=\min_{s\in[0,\overline{u}]}\phi(s),\quad \overline{\phi}_\eta=\max_{s\in [\underline{u}_\eta,\overline{u}]}\phi(s),
$$
and we observe that $ \overline{\phi}_\eta$ is finite for every $\eta$ such that  $0<\eta<x_0-y_0$, even if $\overline{\phi}_\eta$ tends to infinity as $\eta$ tends to $0$.

Then since $\hat{u}\in W^{1,m}(y_0+\eta,x_0)$ and since 
$\phi^{m-1} \in C^0_b([\underline{u}_\eta,\overline{u}]), \mbox{which implies that }   \hat{\psi}_{\delta}\in C^1([\underline{u}_\eta,\overline{u}]),
$
one has the chain rule 
$$
\phi(\hat{u})^{m-1}\frac{d\hat{u}}{dx}=-\hat{\psi}'_\delta(\hat{u})\frac{d\hat{u}}{dx}=-\frac{d\hat{\psi}_{\delta}(\hat{u})}{dx} \ \  \mbox{ in } \ \ L^m(y_0+\eta,x_0),
$$
and therefore one has 
$$
\int_{y_0+\eta}^{x_0}\phi(\hat{u})^{m-1}\frac{d\hat{u}}{dx}dx=\int_{y_0+\eta}^{x_0}-\frac{d\hat{\psi}_{\delta}(\hat{u})}{dx} dx=\hat{\psi}_{\delta}(\hat{u}(y_0+\eta))-\hat{\psi}_{\delta}(\hat{u}(x_0)) \ \  \forall \eta \ \ \text{such that} \ \  0<\eta<x_0-y_0. 
$$
Recall that   $\hat{\psi}_{\delta}(\hat{u}(x_0))$ is  a finite number, while in view of \eqref{414bis} and \eqref{451} one has
$$
\hat{u}(y_0+\eta)\to 0 \ \  \mbox{and} \ \  \hat{\psi}_{\delta}(\hat{u}(y_0+\eta))\to +\infty \ \  \mbox{as} \ \ \eta\to 0,\ \  \eta>0;
$$
this  implies that 
\begin{equation}
\label{452}
\int_{y_0+\eta}^{x_0} \phi(\hat{u})^{m-1}\frac{d\hat{u}}{dx} dx \to +\infty \ \   \mbox{as}\ \  \eta\to 0,\,\eta>0,
\end{equation}
{a contradiction} with \eqref{3000} which implies  that $|\phi(\hat{u})|^{m-1}\in L^{m'}(0,L)$ and  $\dys \frac{d\hat{u}}{dx}\in L^{m}(0,L)$. 

\indent This proves the first part of the claim and therefore Theorem \ref{t3.2}.

\end{proof}
\bk 

The second non-existence result of this subsection is of a different kind. It is concerned with the case of a general singular function $\phi$ which only satisfies \eqref{condphi1} and \eqref{p3} in the case where the datum $g$ is bounded from below, in the framework of weak solutions of problem \eqref{pb1} either in the sense of Definition $1$  (Definition \ref{defin} above) if $m=1$ or in the sense of Definition $m$  (Definition \ref{definm} above) if $m>1$. This result was proved in \cite[Theorem 4.1]{GMMP1} when $m=2$.

\begin{theorem}\label{t3.3} Assume that \eqref{cond1} holds true, and that the data  $(a, g, \phi)$ satisfy \eqref{a1}--\eqref{p3}. 
If there exists a constant $M > 0$ such that the  datum $g\in L^{m}(0,L)$ satisfies
\begin{equation} \label{3.10}
g(x) \ge -M \quad \text{for a.e. } x \in (0,L),
\end{equation}
then it does not exist any weak solution of problem \eqref{pb1} for the data $(a, g, \phi)$ neither in the sense of Definition $1$  (Definition \ref{defin} above)  if $m=1$ nor in the sense of Definition $m$ (Definition \ref{definm} above) if $m>1$.
\end{theorem}

\begin{proof} The proof of Theorem \ref{t3.3}    consists in following along the lines of the proofs of \cite[Lemma 4.3 and Theorem 4.1]{GMMP1}, just   modifying the hypotheses as follows:

 In the proof of Lemma 4.3 of \cite{GMMP1}, one replaces in (4.2) of \cite{GMMP1} the assumption $w \in H^1(0,L)$ by the new assumption $w \in W^{1,1}(0,L)$ (and incidentally the assumption $\phi(w) \in L^2(0,L)$ by the new assumption $\phi(w) \in L^1(0,L)$, but this does not play any role in the proof), and then one  uses the embedding $W^{1,1}(0,L) \subset C^0([0,L])$, which implies that $\phi(w)$ is a continuous function on $[0,L]$ with values in $\mathbb{R}\cup \{+\infty\}$.

\smallskip
Similarly, in the third line of the proof of Theorem 4.1 of \cite{GMMP1}, one replaces  the   assumption $u \in H^1(0,L)$ by the new assumption $u \in W^{1,1}(0,L)$ (and incidentally the assumption $\phi(u) \in L^2(0,L)$ by the new assumption $\phi(u) \in L^1(0,L)$, but this does not play any role in the proof), and then one uses again the embedding $W^{1,1}(0,L) \subset C^0([0,L])$, which implies that $u \in C^0([0,L])$.

\end{proof}

\subsection{Existence of weak solutions of problem \eqref{pb1} in the sense of Definition~$1$   or   of Definition~$m$}
\label{sub3.2}

\mbox{ }
\medskip 

 In this subsection we prove that there exist (large sets of) data $g$ in $L^1(0,L)$ (and in $L^m(0,L)$ for $m>1$) such that problem \eqref{pb1} has a weak solution in the sense of Definition $1$ (Definition \ref{defin} above) (and in the sense of Definition $m$ (Definition \ref{definm} above)). These results are stated and proved only in the model case \eqref{pM} where the singular function $\phi$ is given by
   \[
\phi(s) = \phi_\gamma(s) = \frac{1}{|s|^\gamma} \quad \text{with } \gamma > 0,
\]
when  $m$ and $\gamma$ satisfy either \eqref{3.1}, namely
\begin{equation*}  
 m=1, \ \ \  \gamma>0,
\end{equation*}
or \eqref{3.2}, namely
\begin{equation*}  
m>1, \ \ \ 0 < \gamma < \frac{1}{m-1}.
\end{equation*}

  \noindent Moreover recall (see the introduction of Section \ref{sec3} above) that together with Theorem~\ref{t3.2}, this result proves that in the model case where $\dys \phi(s) = \phi_{\gamma}(s) = \frac{1}{|s|^{\gamma}}$ with $\gamma > 0$, conditions \eqref{3.1} and \eqref{3.2} are actually necessary and sufficient conditions to   ensure the existence of (a large set) of data in $L^1(0, L)$ (and in $L^m(0, L)$ for $m > 1$) for which problem \eqref{pb1} has at least a weak solution in the sense of Definition $1$ (and in the sense of Definition~$m$).

%\begin{remark}
%\label{rem3.1}
%  In \eqref{3.13}, the restriction
%\[
%0<\gamma < \frac{1}{m-1}
%\]
%is suggested  by the non-existence result given in Theorem \ref{t3.2} above. 
%
%\indent Actually,  in view of Theorem \ref{t3.2} and of Theorem~\ref{t3.4} below, this restriction is a necessary and sufficient condition in order to have the existence of data $g \in L^{m}(0,L)$  for which problem \eqref{pb1} has at least a weak solution in the sense of Definition $m$ (Definition \ref{definm} above).
%
%\qed
%\end{remark}

\begin{theorem}\label{t3.4} 
Assume that \eqref{cond1} holds true, and that the data  $(a, g, \phi)$ satisfy \eqref{a1}, \eqref{g1}, and \eqref{pM}, as well as \eqref{3.1} or \eqref{3.2}. Then there exists a large set of data $g\in L^m(0,L)$ for which problem \eqref{pb1} has at least a weak solution in the sense of Definition $1$ (Definition \ref{defin} above)  if $m=1$ or in the sense of Definition $m$ (Definition \ref{definm} above) if $m>1$.
\end{theorem}

This result will be an immediate consequence of Proposition~\ref{glue} below.
\smallskip

Let us first introduce the definition of the set $\mathcal{U}_{m}^{\phi}(0,L)$. 

%We define  the set $\mathcal{G}_{m}^{\phi_{\gamma}}$ of the ``good data'' by
%\begin{equation} \label{3.4}
%\mathcal{G}_{m}^{\phi_{\gamma}} = \left\{ g \in L^{1}(0,L) : \text{at least a solution of problem \eqref{pb1} with data } (a, g, \phi_\gamma) \text{ exists} \right\}
%\end{equation}
%in the sense of Definition \ref{definm} above, 

\begin{defin}
\label{defin3.5}
For $m\geq 1$ and $\phi$ a function which satisfies  \eqref{condphi1} (and possibly \eqref{p3}),  define  the set $\mathcal{U}_{m}^{\phi}(0,L)$ by
\begin{equation} \label{3.3}
\mathcal{U}_{m}^{\phi}(0,L) = \left\{ \hat{u} \in W_{0}^{1,m}(0,L) \text{ such that } \phi (\hat{u}) \in L^{m}(0,L) \right\}. 
\end{equation} \hfill\qed
\end{defin}

 \indent We will naturally denote by $\mathcal{U}_m^{\phi_\gamma}$, $\gamma > 0$, the corresponding set concerned by the case where $\phi$ is the function $\phi_\gamma$ of the model case \eqref{pM}.

\indent A similar set has been  introduced in \cite[(4.13) and (7.2)]{GMMP1},   in the case where $m=2$, where it was  denoted by   $\mathcal{U} $.
\bigskip

It is clear that every weak solution of problem \eqref{pb1}  in the sense of Definition $1$ (Definition \ref{defin} above)  if $m=1$ or in the sense of Definition $m$ (Definition \ref{definm} above) if $m>1$  belongs to $\mathcal{U}_{m}^{\phi}(0,L)$.

Conversely,  for every $\hat{u} \in \mathcal{U}_{m}^{\phi}(0,L)$ and every $\hat{c} \in \mathbb{R}$, define the function $\hat{g}$   by
\begin{equation} \label{3.5}
\hat{g}(x) = a(x)\frac{d\hat{u}}{dx} - \phi(\hat{u}) - \hat{c} \,.
\end{equation}
In view of Remark~\ref{r1.4} above, the function $\hat{u}$ is  a weak solution of problem \eqref{pb1}  for the data $(a, \hat{g}, \phi)$ in the sense of Definition \ref{defin} if $m = 1$ or in the sense of Definition \ref{definm} if $m > 1$. 

Actually, when problem \eqref{pb1} has a weak solution in the sense of Definition \ref{defin} or of Definition \ref{definm} for the data $(a, \hat{g}, \phi)$, then the datum $\hat{g}$ is of the form  \eqref{3.5} for some $\hat{u} \in \mathcal{U}_{m}^{\phi} (0,L)$ and some $\hat{c} \in \mathbb{R}$.

The study of the set $\mathcal{U}_{m}^{\phi}(0,L)$ is therefore essential when studying the existence of weak solutions of problem \eqref{pb1} in the sense of Definition $1$ (Definition \ref{defin} above)  if $m=1$ or in the sense of Definition~$m$ (Definition \ref{definm} above) if $m>1$.

\medskip

\indent In what follows, we will only study the set $\mathcal{U}_m^{\phi_\gamma}$ where $\phi_\gamma$ is the model function given by \eqref{pM}, since we will prove Theorem \ref{t3.4} only in this model case. We will always assume that \eqref{3.1} or \eqref{3.2} holds true. 

Let us now show a way of constructing elements $u$ of $ \mathcal{U}_{m}^{\phi_\gamma}(0,L)$.
A similar construction has been done in [1, Subsection 7.2] in the case where $m=2$.

\indent Define for every $y\in [0,L)$, $\delta^r>0$, $K^r\in \mathbb{R}$, $K^r\not=0$, and $\lambda^r>0$, the function $w^r$ (where $``r"$ stands for ``right") by the formula 
\begin{equation}\label{3.57}
w^r(x)=K^r(x-y)^{\lambda^{r}}\ \ \text{for}\  y\leq x\leq y+\delta^r\,.
\end{equation}
 Since
$$
\frac{dw^{r}}{dx} = K^r \lambda^{r}(x-y)^{\lambda^{r} -1}\ \ \text{and}\ \ \phi_{\gamma}(w^{r})= \frac{1}{|K^{r}|^\gamma (x-y)^{\lambda^{r}\gamma} } \ \ \text{for}\  y\leq x\leq y+\delta^r,
$$
the functions $\displaystyle\frac{dw^{r}}{dx}$ and  $\phi_\gamma(w^{r})$ belong to $L^m(y,y+ \delta^r)$ if and only if one chooses $\lambda^r$ such that 
\begin{equation}\label{lamb}
\frac{m-1}{m}<\lambda^r<\frac{1}{\gamma m}. 
\end{equation}
Such a  choice is possible in view of assumption  \eqref{3.1} or \eqref{3.2}. Let us make such a choice. 

\smallskip 
In  the same way, define for every  $y\in(0,L]$, $\delta^\ell>0$, $K^\ell\in \mathbb{R}$, $K^\ell\not=0$, and $\lambda^\ell>0$, the function $w^\ell$ (where $``\ell"$ stands for ``left") by the formula 

\begin{equation}\label{3.6}
w^{\ell}(x)=K^{\ell}(y-x)^{\lambda^{\ell}}\ \ \text{for}\  y-\delta^\ell\leq x\leq y\,.
\end{equation}
Then the functions $\displaystyle\frac{dw^{\ell}}{dx}$ and  $\phi_{\gamma}(w^{\ell})$ belong to $L^m(y-\delta^\ell,y)$  if and only if one chooses $\lambda^\ell$ such that  
\begin{equation}\label{lambl}\frac{m-1}{m}<\lambda^\ell<\frac{1}{\gamma m}. \end{equation} 
  
  \noindent Let us make such a choice. 
  
  Take now   $x_1$, $x_2$ and $\delta_1^r>0$, $\delta_2^\ell>0$ such that
\begin{equation}
\label{3.10bis}
0 \le x_1 < x_1+\delta_1^r < x_2-\delta_2^\ell < x_2 \le L,
\end{equation}
and denote by $w_1^r$ the function defined by \eqref{3.57} with $y=x_1$, $\delta^r=\delta^r_1$, $K^r=K^r_1$, and $\lambda^r=\lambda^r_1$, and by $w_2^\ell$ the function defined by \eqref{3.6} with $y=x_2$, $\delta^\ell=\delta^\ell_2$, $K^{\ell}=K^{\ell}_2$, and $\lambda^{\ell}=\lambda^{\ell}_2$.

Assume now that \begin{equation}\label{3.211}K_{1}^r K_{2}^\ell > 0,\end{equation} or in other terms,  that $K_1^r$ and $K_2^\ell$ (and also $w_1^r$ and $w_2^\ell$) have the same sign.

Take then any function $w^{int}_{1,2} \in W^{1,m}(x_1+\delta_1^r, x_2-\delta_2^\ell)$ which satisfies
\begin{equation}\label{3.212}
w_{1,2}^{int}(x_1+\delta_1^r) = w_1^r(x_1+\delta_1^r) = K_1^r {\delta_1^{r}}^{\lambda_1^{r}} \ \ \text{and}\ \ 
w_{1,2}^{int}(x_2-\delta_2^\ell) = w_2^\ell(x_2-\delta_2^\ell) = K_2^\ell {\delta_2^{\ell}}^{\lambda_2^{\ell}},
\end{equation}
as well as, for some $\eta_{1,2} > 0$,
 \begin{equation}\label{3.213}|w_{1,2}^{int}(x)| \ge \eta_{1,2} > 0\,\,\,\,  \forall x\in [x_1+\delta_1^r, x_2-\delta_2^\ell].\end{equation}

\indent Define finally on the segment $[x_1,x_2]$  the function $w_{1,2}$ by
\begin{equation} \label{3.9}
w_{1,2}(x) =
\begin{cases}
w_1^r(x) & \text{if } x \in [x_1, x_1+\delta_1^r], \\
w^{int}_{1,2}(x) & \text{if } x \in [x_1+\delta_1^r, x_2-\delta_2^\ell], \\
w_2^\ell(x) & \text{if } x \in [x_2-\delta_2^\ell, x_2],
\end{cases}
\end{equation}
 and, for any given constant $c_{1,2}\in\mathbb{R}$, define on the segment $[x_1,x_2]$ the function $h_{1,2}$ by
\begin{equation}
\label{3.14bis}
h_{1,2}(x)=a(x)\frac{ dw_{1,2}}{dx}-\phi_{\gamma}(w_{1,2})-c_{1,2}\, \ \text{ a.e. }\ x\in (x_1,x_2).  
\end{equation}
It is easy  to check that  $w_{1,2}\in  \mathcal{U}_{m}^{\phi_\gamma}(x_1,x_2)$ and that $h_{1,2}\in L^m(x_1,x_2)$. 

\medskip

Note that in this construction we have a large liberty in the choices of the points $x_1$ and $x_2$, of the parameters $\delta_1^r$ and $\delta_2^\ell$, of the constants $K_1^r$ and $K_2^{\ell}$, of the powers $\lambda_1^r$ and $\lambda_2^{\ell}$, of the parameter $\eta_{1,2}$, of the function $w^{int}_{1,2}$, and finally of  the constant $c_{1,2}$. 

\medskip
For $n\in\mathbb{N}$ fix now $n+2$ points $x_0, x_1,\, \cdots\,, x_n, x_{n+1}$ such that
\begin{equation} \label{3.1013}
0 = x_0 < x_1 <  \,\cdots\, < x_n < x_{n+1} = L, 
\end{equation}
and choose on each segment $[x_j, x_{j+1}]$ for $j=0, 1, \,\cdots \,, n,$   two functions $ {w}_{j,j+}1$ and ${h}_{j,j+1}$ by the construction we have just made, with parameters $\delta_j^r$ and $\delta_{j+1}^\ell$ which satisfy \eqref{3.10bis}, with constants $K_j^r$ and $K_{j+1}^\ell$ which satisfy \eqref{3.211}, with powers $\lambda_j^r$ and $\lambda_{j+1}^\ell$ which satisfy \eqref{lamb} and \eqref{lambl}, with parameters $\eta_{j,j+1}$ which satisfy $\eta_{j,j+1}>0$, with functions $w_{j,j+1}^{int}$ which satisfy \eqref{3.212} and \eqref{3.213}, and finally with constants $c_{j,j+1}\in\mathbb{R}$. Then  ``glue together" the segments $[x_j, x_{j+1}]$  as well as the functions $w_{j,j+1}$ and $h_{j,j+1}$,  $j=0,1,\, \cdots\, , n $  to construct on $[0,L]$ two functions $\hat{u}\in W_0^{1,m}(0,L)$ and $\hat{g}\in L^m(0,L)$. We have proved the following  result.

\begin{proposition}\label{glue} Assume that \eqref{cond1} holds true, and that the data  $(a, g, \phi)$ satisfy \eqref{a1}, \eqref{g1}, and \eqref{pM}, as well as \eqref{3.1} or \eqref{3.2}. For $n \in \mathbb{N}$ fix $n+2$ points $x_k$ which satisfy \eqref{3.1013}. 
Then there exists a large set of triplets $(\hat{u}, \hat{g},\hat{c})$ which satisfy
\begin{equation}\begin{cases}
      \hat{u} \in \mathcal{U}_m^{\phi_\gamma}(0,L),\\
   \hat{u}(x_k) = 0 \quad\forall  k=0, \,\cdots\,, n+1,\footnotemark\\
   \hat{g}\in L^{m}(0,L),\\
   \dys \hat{c}\in\mathbb{R},\\
 \dys\hat{g}(x)= a(x)\frac{d\hat{u}}{dx} - \phi_{\gamma}(\hat{u})-\hat{c}   \   \ \text{in}\ \mathcal{D}'(0,L),
   \end{cases}
   \end{equation}\footnotetext{Note also, incidentally, that $\hat{u}$ satisfies $\hat{u}(x)\neq 0$ for every $x\neq x_k$, $k=0,\, \cdots\,,n+1$.}\noindent or in other words a large set of functions $\hat{u}$ which are such that   $\hat{u}$ is a weak solution of problem \eqref{pb1} for the data $(a, \hat{g}, \phi_\gamma)$ in the sense of Definition $1$ (Definition \ref{defin} above)  if $m=1$ or in the sense of Definition $m$ (Definition \ref{definm} above) if $m>1$. 
\end{proposition}
% and that $w$ is a solution of
%$$
%\begin{cases}
%\dys w\in W^{1,m}_0(x_1,x_2), \,\,\,\, \frac{1}{|w|^\gamma}\in L^m(x_1,x_2), \\
%\dys -\frac{ d}{dx}\left(a(x)\frac{ dw}{dx}\right)  = - \frac{ d \phi_\gamma (w)}{dx} 
% - \frac{ d h}{dx}  \,\,\,\, \text{in}\,\,\mathcal{D}'(x_1,x_2)\,.
%\end{cases}
%$$
\smallskip

\subsection{Every weak solution of problem \eqref{pb1} in the sense of Definition~$1$ or  of Definition~$m$  is not~isolated}
\label{sub3.3}
\mbox{ }
\smallskip

In this subsection, we prove that every weak solution $u$ of problem \eqref{pb1} in the sense of Definition~$1$ (Definition \ref{defin} above) or in the sense of Definition $m$ (Definition \ref{definm} above)  is not isolated: in other terms, we prove that if $u$ is any weak solution of problem \eqref{pb1} in the sense of Definition $1$ (Definition \ref{defin} above) or in the sense of Definition $m$ (Definition \ref{definm} above) for some data $(a,g,\phi)$, then there exists a sequence of data $(a,g_n,\phi_n)$ for which there exists at least a weak solution $u_n$ of problem \eqref{pb1} in the sense of Definition $1$ (Definition \ref{defin} above) or in the sense of Definition $m$ (Definition \ref{definm} above) which is such that $u_n\not\equiv u$ for every $n$, such that the sequence $u_n$ strongly converges to $u$ in $W_0^{1,1}(0,L)$ or in $W_0^{1,m}(0,L)$, and such that the sequence of data $g_n$ strongly converges to $g$ in $L^1(0,L)$ or in $L^m(0,L)$; moreover the sequence $\phi_n$ can be any a priori given sequence of reasonable approximations of $\phi$ (see Definition~\ref{3.7} below for the definition of this notion)  while the sequence of data $g_n$ should be chosen accordingly. 
\smallskip

The definition of the notion of a reasonable sequence of approximations of a function $\phi$ has been introduced in \cite[Definition~2.16]{GMMP1} and does not depend on the parameter $m$. It reads as follows: 
\begin{defin}
\label{3.7}
Let $\phi$ be a function which satisfies \eqref{condphi1}. We will say that a sequence $\phi_n$  is a {\it sequence of reasonable approximations of $\phi$} if the sequence $\phi_n$ satisfies
\begin{equation}
\label{3.42}
\phi_n \mbox{ satisfies assumption \eqref{condphi1} for every given $n$},
\end{equation}
\begin{equation}
\begin{cases}
\label{3.43}
\mbox{for every sequence $s_n\in\re$ and every $s\in\re$ such that $s_n\to s$ in $\re$, one has}\\
\phi_n(s_n)\to\phi(s) \mbox{ in } \re\cup\{+\infty\}.
\end{cases}
\end{equation}
 \hfill\qed
\end{defin}

The model example of a sequence of reasonable approximations of a given function $\phi$ is the approximation of $\phi$ by truncation defined by
$$
\phi_n(s)=T_n(\phi(s))\quad \forall s\in\re\quad \forall n\in\mathbb{N},
$$
where $T_n:\re\mapsto\re$ is the classical truncation function  at height $n$ defined  by
\begin{equation}
\label{3.44}
T_n(s)=
\begin{cases}
\dys s & \mbox{if} \,\, |s|\leq n,\\
\dys n\frac{s}{|s|} &\mbox{if} \,\, |s|\geq n; 
\end{cases}
\end{equation}
for other examples of sequences of reasonable approximations (the homographic approximation, the trivial approximation, the approximation by convolution, ...), and for an equivalent definition of a reasonable sequence of approximations, see \cite[Subsection~2.4]{GMMP1};  note finally that for every function $\phi$ which satisfies \eqref{condphi1} there exist sequences $\phi_n$ of reasonable approximations of $\phi$ which satisfy the domination condition
\begin{equation}
\label{3.18}
\exists\, C\in\mathbb{R} \quad \mbox{such that}\quad |\phi_n(s)|\leq C|\phi(s)|\,\,\,\, \forall s\in\re\,\,\,\, \forall n\in\mathbb{N};
\end{equation}
two examples of reasonable approximations which satisfy \eqref{3.18} with $C=1$ are the approximations by truncation and the homographic approximations, see \cite[(2.44), (2.45), and (2.46)]{GMMP1}.
\smallskip

We can now state and prove the result which mathematically expresses that any weak solution of problem \eqref{pb1} in the sense of Definition $1$ (Definition \ref{defin} above) or in the sense of Definition $m$ (Definition \ref{definm} above) is not isolated. 

\begin{proposition}\label{stab}
Assume that \eqref{cond1} holds true, and that  the data $(a,g,\phi)$ satisfy hypotheses \eqref{a1}--\eqref{condphi1}.  Assume also that there exists a  weak  solution $u$ of problem \eqref{pb1} in the sense of Definition $1$ (Definition \ref{defin} above) if $m=1$ or in the sense of Definition $m$ (Definition \ref{definm} above) if $m>1$, or equivalently assume (see Remark~\ref{r1.4} above) that there exists a function $u$ which satisfies 
\begin{equation}
\begin{cases}\label{71}
u\in W^{1,m}_0(0,L), \, \,\,\,  \phi(u)\in L^m(0,L), \\\\
\dys -\frac{d}{dx}\left(a (x) \frac{du}{dx}\right)  = - \dys\frac{d \phi (u)}{dx} - \frac{dg(x)}{dx} \ \  \text{in}\,\;\mathcal{D}'(0,L).
\end{cases}
\end{equation}
\indent Fix any  reasonable sequence of  approximations $\phi_n$ of $\phi$,     and assume that this sequence $\phi_n$ satisfies the domination condition \eqref{3.18}. 
Then there exists a sequence $g_n$  which satisfies 
\begin{equation}
\label{81ter}
\text{$g_n\in L^m(0,L)$,   \,   $g_n\to g $ strongly in $L^m(0,L)$,}
\end{equation}
  for which  there exists a weak  solution $u_n$ of problem \eqref{pb1} for the data $(a,g_n,\phi_n)$   in the sense of Definition $1$ (Definition \ref{defin} above) if $m=1$ or in the sense of Definition $m$ (Definition \ref{definm} above) if $m>1$, or equivalently (see again Remark~\ref{r1.4} above)  for which there exists a solution $u_n$ of  
\begin{equation}
\begin{cases}\label{trivial}
u_n\in W^{1,m}_0(0,L), \, \,\,\,  \phi_n(u_n)\in L^m(0,L),\\\\
\dys -\frac{d}{dx}\left(a (x) \frac{du_n}{dx}\right)  = - \dys\frac{d \phi_n (u_n)}{dx} - \frac{d g_n(x)}{dx} \ \  \text{in}\,\,\mathcal{D}'(0,L)\,,
\end{cases}
\end{equation}
which satisfies
\begin{equation}
\label{3.20bis}
u_n\not\equiv u \quad \forall n,
\end{equation}
\begin{equation}
\label{3.20ter}
u_n \rightarrow u \ \ \text{strongly in}\ \ W^{1,m}_0(0,L).
\end{equation}
\end{proposition}

In the case where $m=2$, Proposition~\ref{stab} has been proved  in \cite[Proposition~8.1 and Remark~8.2]{GMMP1}.

\begin{proof}[\bf Proof]
The proof of Proposition~\ref{stab} follows along the lines of the proof of \cite[Proposition~8.1]{GMMP1}, and surprisingly consists in defining first the solution $u_n$ and then the datum $g_n$.
\smallskip

For $u$ which satisfies \eqref{71}, fix some $x_0\in (0,L)$ with $|u(x_0)|>0$ and some interval $[A,B]\subset (0,L)$ such that for some $\eta>0$, one has 
\begin{equation}
\label{3.60}
0<A<x_0<B<L,\quad |u(x)|\geq \eta
\quad \forall x\in [A,B].
\end{equation}

Take also any sequence $z_n$ such that 
\begin{equation}
\label{3.61}
z_n\not\equiv0,\,\, z_n\in W_0^{1,m}(0,L),\,\, z_n\rightarrow 0 \quad \mbox{strongly in } \,\,  W_0^{1,m}(0,L),\,\, z_n(x)=0\quad \forall x\in [0,A]\cup [B,L].
\end{equation}

Since 
$$
|z_n(x)|= \left|\int_0^x \frac{dz_n(t)}{dx}\, dt\right|\leq \left\| \frac{dz_n}{dx}\right\|_{L^1(0,L)} \,\, \forall x\in[0,L],
$$
one has 
$$
z_n(x)\rightarrow 0 \quad \mbox{uniformly on}\,\, (0,L),
$$
so that one can assume without any loss of generality that
\begin{equation}
\label{3.63}
\|z_n\|_{C^0([0,L])}\leq \frac{\eta}{2} \quad \forall n\in \mathbb{N}.
\end{equation}

Define $u_n$ by  
\begin{equation}
\label{3.64}
u_n=u+z_n.
\end{equation}

Then in view of \eqref{3.61} and \eqref{3.64} the sequence $u_n$ satisfies \eqref{3.20bis} and \eqref{3.20ter}.

We now claim that 
\begin{equation}
\label{3.65}
\phi_n(u_n)\rightarrow \phi(u) \quad \mbox{strongly in } \,\,  L^m(0,L).
\end{equation}

Assuming for a moment that the claim \eqref{3.65} holds true, and recalling the definition of the constant $c$  from $u$, $\phi(u)$, and $g$ given by  \eqref{definC},    namely 
% from $u$, $\phi(u)$, and $g$ given by \eqref{definODE}  and \eqref{definC},   
$$
\dys c=-\frac{\dys \int_0^L \frac{\phi(u)}{a(x)}dx+\int_0^L \frac{g(x)}{a(x)}\, dx}{\dys\int_0^L \frac{1}{a(x)}\, dx},
$$ 
 we define    $g_n$ by 
\begin{equation}
\label{3.66}
\dys g_n(x)=a(x)\frac{du_n}{dx}-\phi_n(u_n)-c\ \ \text{in}\  \mathcal{D}'(0,L).
\end{equation}
\smallskip

Then $u_n$ is a solution of \eqref{trivial}, and therefore, in view of Remark~\ref{r1.4} above, $u_n$ is a weak solution of problem \eqref{pb1} for the data $(a,g_n,\phi_n)$ in the sense of Definition $1$ (Definition \ref{defin} above) if $m=1$ or in the sense of Definition $m$ (Definition \ref{definm} above) if $m>1$. Moreover, in view of \eqref{3.66}, \eqref{3.20ter}, \eqref{3.65}, \eqref{definODE}, and \eqref{definC}, the sequence $g_n$ satisfies \eqref{81ter}.
\smallskip

The proof of Proposition~\ref{stab} will be complete, once the claim \eqref{3.65} has been proved. Let us do it now.

Since the sequence $\phi_n$ is a reasonable sequence of approximations of $\phi$, property \eqref{3.43} with \eqref{3.64} and \eqref{3.61} imply that 
\begin{equation}
\label{3.67}
\phi_n(u_n(x))\rightarrow \phi(u(x)) \,\,\,\, \forall x\in [0,L].
\end{equation}

On the other hand, the domination condition \eqref{3.18} implies that 
\begin{equation}
\label{3.68}
|\phi_n(u_n(x))|\leq C\, |\phi(u_n(x))| \,\,\,\, \forall x\in [0,L] \,\,\,\,\,  \forall n\in \mathbb{N}.
\end{equation}

On  the set $(0,A)\cup(B,L)$, the property that $u_n\equiv u$ (see \eqref{3.61}  implies that $\phi_n(u_n(x))$ is dominated by the fixed function $C|\phi(u(x))|$ which belongs to $L^m((0,A)\cup (B,L))$.

On the set $[A,B]$, in view of \eqref{3.64}, \eqref{3.63}, and \eqref{3.60}, one has
$$
|u_n(x)|=|u(x)+z_n(x)|\leq \|u\|_{C^0([A,B])} +\frac{\eta}{2}\,\,\,\,\, \forall x\in[A,B] \,\,\,\,\, \forall n\in \mathbb{N},
$$

$$
|u_n(x)|=|u(x)+z_n(x)|\geq \frac{\eta}{2}\,\,\,\, \forall x \in[A,B]\,\,\,\,\, \forall n\in \mathbb{N},
$$
and therefore  
$$
\frac{\eta}{2}\leq |u_n(x)|\leq \|u\|_{C^0([A,B])} +\frac{\eta}{2}\,\,\, \forall x\in[A,B] \,\,\, \forall n\in \mathbb{N}.
$$

But in view of \eqref{condphi1}, the function $\phi$ belongs to $C^0([\delta,R])\cup C^0([-R,-\delta])$ for every $\delta$ and $R$ such that $0<\delta<R$. Therefore the function $\phi_n(u_n)$, which in view of \eqref{3.68} is dominated on $[A,B]$ by $C\,|\phi(u_n)|$, is also dominated on $[A,B]$ by the constant 
$
\dys\max\left\{C\,|\phi(s)|;\,\, \frac{\eta}{2}\leq |s|\leq \|u\|_{C^0([A,B])}+\frac{\eta}{2}\right\}
$
which belongs to $L^m(A,B)$.

The dominations of $\phi_n(u_n)$ on $(0,A)\cup(B,L)$ and on $(A,B)$ and the almost everywhere convergence \eqref{3.67} of $\phi_n(u_n)$ together with Lebesgue's dominated convergence Theorem then complete the proof of the claim \eqref{3.65} and of Proposition~\ref{stab}.

\end{proof}
 \vspace{1cm}

\section{Comments and remarks about the present paper and about \cite{GMMP1}}

The aim of the present  paper was to study from a new angle the problem that we treated in \cite{GMMP1}. 

In \cite{GMMP1} we have introduced a notion of weak solution of problem \eqref{pb1} which is recalled in Definition $2$ (Definition \ref{defin2} above) of the present paper; this definition is based on the space $L^2(0,L)$. 

% [andare a 2 di pag. 14]

In the present  paper we have given two new definitions of  weak solution of problem \eqref{pb1}: Definition~$1$ (Definition \ref{defin} above) based on the space $L^1(0,L)$ and Definition $m$ (Definition \ref{definm} above) based on the space $L^m(0,L)$, and we have stated and proved in Section~\ref{sec3} above some properties of these weak solutions: non-existence of weak solutions (Subsection~\ref{sub3.1} above),  existence of weak solutions (Subsection~\ref{sub3.2} above),  and non-isolation of the weak solutions (Subsection~\ref{sub3.3} above). 

Surprisingly, some of these results depend on the definition which is used, or more precisely on the value of the parameter $m$. The most striking illustration of this fact is the result which states that in the model case \eqref{pM} of the singular function $\phi$, namely in the case where $\dys\phi(s)=\phi_\gamma(s)=\frac{1}{|s|^\gamma}$ with $\gamma>0$, when one uses Definition $m$ (Definition \ref{definm} above) with $m>1$, there exists (a large set of) data $g\in L^m(0,L)$ for which problem \eqref{pb1} has a weak solution if and only if $\gamma$ satisfies $\dys 0<\gamma<\frac{1}{m-1}$ (see Theorem~\ref{t3.2}, Theorem~\ref{t3.4}, and the introduction of Section \ref{sec3} above), and also that when one uses Definition $1$ (Definition \ref{defin} above)   with $m=1$, it exists (a large set of)  data $g\in L^1(0,L)$ for which problem \eqref{pb1} has a weak solution for every $\gamma>0$.

\medskip

The new approach to problem \eqref{pb1} developped in the present paper and based on Definition $1$ (Definition \ref{defin} above) and on Definition $m$ (Definition \ref{definm} above) is a continuation and an extension of the approach used in \cite{GMMP1}. 
Actually in the present paper we recover for $m \ge 1$ some (variants of the) results that we obtained in \cite{GMMP1}, but some other ones are missing.

\medskip

Between those missing results is the energy equality of \cite[Proposition~2.10]{GMMP1}, which is recalled in \eqref{vf2bis} above. This energy equality obtained in the case $m=2$ is particularly important since it allows one to obtain in this setting an a priori estimate in $H_0^1(0,L)$ for any  possible weak solution, even if this a priori estimate (which is also available in dimension $N>1$) is very far from being sufficient to obtain the existence of weak solutions. In dimension $N=1$,  still in the case $m=2$, we were able  to obtain in \cite[Lemma 3.1]{GMMP1} a specific new estimate on the approximations  $\phi_n(u_n)$ which allowed us to prove an alternative (see \cite[Theorem~3.4]{GMMP1})\footnote{Note that we were unable to obtain any similar result in the present paper in the framework  of Definition $1$ (Definition \ref{defin} above) and in the framework  of Definition $m$  (Definition \ref{definm} above) when $m\not=2$.}. This alternative states that for any ``reasonable approximation" (see Definition~\ref{3.7} above) of problem \eqref{pb1} one can pass to the limit in a subsequence of solutions of these approximate problems, for which one has the following two possibilities: either the limit $u$ is a weak solution of problem \eqref{pb1} in the sense of Definition $2$ (Definition \ref{defin2} above), or $u\equiv 0$. This result looks like an existence result, but actually is not at all an existence result: in fact in \cite{GMMP1} we were unable to describe a set of data $g$ in $L^{2}(0,L)$\footnote{ Note that we were also unable to define such a set of data in the present paper in the framework of Definition $1$ and  of Definition  $m$.}  such that, for $g$ in this set,   problem \eqref{pb1} has a weak solution in the sense of Definition~\ref{defin2} for the  data $(a,g,\phi)$. The only set $\mathcal{G}$ of data for which we were able to prove an existence result   in \cite{GMMP1} is actually  tautological since it consists in  the image (by the nonlinear singular operator of problem~\eqref{pb1}) of the set $\mathcal{U}$ defined in \cite[formula (7.1)]{GMMP1} of those functions $u$ such  that $u\in H_0^1(0,L)$ for which $\phi(u)\in L^2(0,L)$ (see \cite[Remark~7.1]{GMMP1}). 

%In the present paper, we were not able to obtain in the framework of Definition $1$ (Definition \ref{defin} above) and Definition $m$ (Definition \ref{definm} above)  better results that those that we have obtained in \cite{GMMP1} in the framework of Definition $2$ (Definition \ref{defin2} above): we were only able to prove (see Subsection~\ref{sub3.2} above, and in particular Theorem~\ref{t3.4} and the proof of Proposition~\ref{glue} above) that in the frameworks of Definition $1$ (Definition \ref{defin} above) and Definition $m$ (Definition \ref{definm} above), in the model case where $\dys\phi(s)=\phi_\gamma(s)=\frac{1}{|s|^\gamma}$ with $m$ and $\gamma$ satisfying \eqref{3.1} or \eqref{3.2}, there exists a large set of data $g\in L^m(0,L)$ for which problem \eqref{pb1} has a weak solution in the sense of  Definition $1$ (Definition \ref{defin} above) and Definition $m$ (Definition \ref{definm} above): like in \cite[Section~7]{GMMP1} in the proof of Proposition~\ref{glue} above we have constructed ``by hand" (i.e. without using functional analysis methods)  a large set of elements $\hat{u}$ on  the set $\mathcal{U}_{m}^{\phi_{\gamma}}$, which are indeed weak solutions of problem \eqref{pb1} in the sense of Definition $1$ (Definition \ref{defin} above) and in the sense of Definition $m$ (Definition \ref{definm} above) once one has conveniently defined the corresponding data $g$.

\medskip 

Other missing results, still due to the lack of any a priori estimate, are results of stability and instability of  the  solutions of problems approximating problem \eqref{pb1}  analogous  to the results given  in \cite[Section 8]{GMMP1};    we were unable to obtain such results  in the frameworks  of Definition $1$ (Definition~\ref{defin} above) or    of Definition $m$ (Definition~\ref{definm} above).    
%,  in the frameworks of ,

\medskip 

A third set of results which have been obtained in \cite{GMMP1} in the framework of Definition $2$ (Definition \ref{defin2} above), but which have no analogues in the present paper in the framework of Definition $1$ (Definition \ref{defin} above) and of  Definition $m$ (Definition \ref{definm} above), is the set of results concerning the ordinary differential equation (ODE)  \eqref{ODE},  and more precisely the results that we obtained in \cite{GMMP1}  concerning existence, a priori estimates, stability, positivity, comparison and uniqueness of solutions of the ODE \eqref{ODE}.   

These  results on the ODE \eqref{ODE}   have had important consequences in the study of the possible weak solutions of problem \eqref{pb1} in the sense of Definition $2$. Indeed they allowed us, on the first hand, to prove a surprising multiplicity result (see \cite[Theorem 6.1]{GMMP1}) and, on the second hand, to construct, for any given datum $g \in L^2(0,L)$ and $\delta$ satisfying $0 < \delta < L$, a datum $\hat{g}_\delta \in L^2(0,L)$ which coincides with $g$ on the interval  $[0, L-\delta]$  for which there exists a weak solution of problem \eqref{pb1} in the sense of Definition $2$ (Definition \ref{defin2} above). In the absence of results for  the analogues of ODE \eqref{ODE} with  $m \ge 1$ ($m \ne 2$), we are unable to prove such results in the framework of Definition $1$ and of Definition $m$ with $m \ge 1$ ($m \ne 2$). We nevertheless hope that such results could be proved and we will return on this problem in a future work.

\vspace{1cm}
\section*{Acknowledgements}

 The   authors were partially supported by  the Gruppo Nazionale per l'Analisi Matematica, la Probabilit\`a e le loro Applicazioni (GNAMPA) of the Italian Istituto Nazionale di Alta Matematica (INdAM). The second author was also supported by the FQM-424 Spanish Research Group,  and by CDTIME. The authors would like to warmly thank the Editors of this Special Issue for dedicating it    to  Gioconda and for  asking  them to contribute to it.


\vspace{1cm}


\begin{thebibliography}{99}

% \bibitem{AMM} D. Arcoya and L. Moreno-M\'erida, {\it Multiplicity of solutions for a Dirichlet problem with a strongly singular nonlinearity}, Nonlinear Anal. {\bf 95} (2014), 281--291.
% 	\bibitem{a6}
%D. Arcoya, J. Carmona, T. Leonori, P. J. Mart\'nez-Aparicio, L. Orsina and F. Petitta,
%{\it Existence and nonexistence of solutions for singular quadratic quasilinear equations}, 
%J. Differential Equations  {\bf 246}   (2009), 4006--4042. 
%
%
%\bibitem{b4} D. Arcoya, L. Boccardo, T. Leonori and A. Porretta, 
%{\it Some elliptic problems with singular natural growth lower order terms}, 
%J. Differential Equations {\bf 249} (11) (2010), 2771--2795.  
%     \bibitem{bop} F. Balducci, F. Oliva and F. Petitta,  {\it Finite energy solutions for nonlinear elliptic equations with competing gradient, singular and $L^1$ terms}, J. Differential Equations  {\bf 391} (2024) , 334--369. 
%     
%     \bibitem{bo1} L. Boccardo, {\it Dirichlet problems with singular and gradient quadratic
%lower order terms}, ESAIM: Control, Optimization and the Calculus
%of Variations, {\bf 14} (2008) 411--426. 
%
%
% \bibitem{BC} L. Boccardo and J. Casado-D\'iaz, {\it Some properties of solutions of some semilinear elliptic singular problems and applications to the  $G$-convergence}, Asymptot. Anal. {\bf 86} (1) (2014),  1--15. 
% 
% \bibitem{BGDM2} L. Boccardo, J. I.  Diaz, D. Giachetti and F.  Murat, {\it Existence of a solution for a weaker form of a nonlinear elliptic equation}, In: "Pitman Res. Notes Math. Ser.", 208, Longman Scientific \& Technical, Harlow,  1989, 229--246. 
%
% 
%\bibitem{BGDM} L. Boccardo, D. Giachetti, J. I.  Diaz and F.  Murat,
% {\it Existence and regularity of renormalized solutions for some elliptic problems involving derivatives of nonlinear terms}, J. Differential Equations {\bf 106} (2) (1993), 215--237.
%
%
%
% 	\bibitem{BO} L. Boccardo and L. Orsina, {\it Semilinear elliptic equations with singular nonlinearities},  Calc. Var. Partial Differential Equations {\bf 37} (2010), 363--380.
%	 \bibitem{CST} {A. Canino, B. Sciunzi and A. Trombetta},
%{\it Existence and uniqueness for {$p$}-Laplace equations
%              involving singular nonlinearities.}
%NoDEA Nonlinear Differential Equations Appl. 
% {\bf 23} (2016), 8--18. 
%
% 
%  \bibitem{CMA} J. Carmona and P. J. Mart\'inez-Aparicio, {\it A singular semilinear elliptic equation with a variable exponent},  Adv. Nonlinear Stud. {\bf 16} (2016), 491--498. 
%  
% \bibitem{CD-M} J. Casado-D\'iaz and F. Murat, {\it Semilinear problems with right-hand sides singular at  $u=0$  which change sign}, 
%Ann. Inst. H. Poincar\'e C Anal. Non Lin\'eaire 38 (3) (2021), 877--909. 
% 
% \bibitem{Crandall} M.G. Crandall, P.H. Rabinowitz and L. Tartar, {\it On a dirichlet problem with a singular nonlinearity},  Comm. Partial Diff. Equations, \textbf{2} (1977), 193--222.
% 
% \bibitem{D-G}  P. Donato and D. Giachetti, {\it Existence and homogenization for a singular problem through rough surfaces},
%SIAM J. Math. Anal. {\bf 48} (6) (2016),  4047--4086.
% 
%\bibitem{Fulks-Maybee} W. Fulks and J. S. Maybee,  {\it A singular non-linear equation}, 
%Osaka Math. J. {\bf 12} (1960), 1--19. 
%
%	\bibitem{GMM4} D. Giachetti, P. J. Mart\'inez-Aparicio and F. Murat,  {\it Advances in the study of singular semilinear elliptic problems}, 
%In: "Trends in differential equations and applications", SEMA SIMAI Springer Ser., 8, Springer, [Cham], 2016,  221â241.
% 
%\bibitem{GMM2} D. Giachetti, P. J. Mart\'inez-Aparicio and F. Murat, {\it A semilinear elliptic equation with a mild singularity
%at $u = 0$: Existence and homogenization},   J. Math. Pures Appl. {\bf 107} (2017),  41--77.
% 
%\bibitem{GMM} D. Giachetti, P. J. Mart\'inez-Aparicio and F. Murat, 
%	{\it Definition, existence, stability and uniqueness of the solution to a semilinear elliptic problem with a strong singularity at $ u = 0 $,} Ann. Scuola Normale Pisa (5) {\bf 18} (4) (2018), 1395--1442. 
%	
%	\bibitem{GMM1} D. Giachetti, P. J. Mart\'inez-Aparicio and F. Murat, {\it Homogenization of a Dirichlet semilinear elliptic problem with a strong singularity at  $u=0$  in a domain with many small holes},   
%J. Funct. Anal. {\bf 274} (6) (2018), 1747--1789.
%\bibitem{GMM3}  D. Giachetti, P. J. Mart\'inez-Aparicio and F. Murat, {\it On the definition of the solution to a semilinear elliptic problem with a strong singularity at  $u=0$}, 
%Nonlinear Anal. {\bf 177} (2018), 491--523. 	
%
%%\bibitem{GMMP1}  D. Giachetti, P. J. Mart\'inez-Aparicio, F. Murat and F. Petitta, {\it Remarks and variants on a one-dimensional elliptic equation with a singular first order divergence term}, to appear
%
%\bibitem{GMMP2}  D. Giachetti, P. J. Mart\'inez-Aparicio, F. Murat and F. Petitta, {\it A   one-dimensional elliptic equation with a  first order divergence term singular in $m\neq 0$}, to appear


\bibitem{GMMP1} D. Giachetti, P.J. Mart\'i­nez-Aparicio, F. Murat, and   F. Petitta, {\it Unexpected phenomena in a one-dimensional elliptic equation with a singular first order divergence term}, Ann. Sc. Norm. Sup. Pisa Cl. Sci. (5), in press. 


% \bibitem{Lair-1} A. V. Lair and A. W. Aihua,  {\it Entire solution of a singular semilinear elliptic problem},  J. Math. Anal. Appl. {\bf 200} (1996), 498--505. 
% 
% \bibitem{Lair-2} A. V. Lair and A. W. Aihua, {\it Classical and weak solutions of a singular semilinear elliptic problem}, J. Math. Anal. Appl. {\bf 211} (1997), 371--385. 
%
%
%     \bibitem{LM} A. C. Lazer and P. J. McKenna,  {\it On a singular nonlinear elliptic boundary-value problem}, Proc. Amer. Math. Soc. {\bf 111} (1991),  721--730. 
%     
%     \bibitem{pe} P. J. Mart\'inez-Aparicio, {\it Singular Dirichlet problems with quadratic gradient}. 
%Boll. Unione Mat. Ital. (9) {\bf 2} (3) (2009), 559--574.
%     
%  \bibitem{OP1} F. Oliva and F. Petitta,  {\it On singular elliptic equations with measure sources}, ESAIM Control Optim. Calc. Var. \textbf{22}  (1)  (2016),  289--308.  
%\bibitem{OP} F. Oliva and F. Petitta, {\it Finite and infinite energy solutions of singular elliptic problems: Existence and uniqueness},   J. Differential Equations \textbf{264} (1) (2018),   311--340.  
%
%\bibitem{op} F. Oliva, F. Petitta, {\it Singular Elliptic PDEs: an extensive overview},  Partial Differential Equations and Applications   \textbf{6} (6) (2025). 
%
%\bibitem{Stuart}  C. A. Stuart, {\it Existence and approximation of solutions of non-linear elliptic equations},
%Math. Z. {\bf 147}  (1) (1976),   53--63. 
%
%\bibitem{Z} Z. Zhang and J. Cheng,  {\it Existence and optimal estimates of solutions for singular nonlinear Dirichlet problems}, Nonlinear Anal. {\bf 57} (2004), 473--484. 
%
\end{thebibliography}
\end{document}